\documentclass{amsart}
\usepackage{}
\usepackage{amsmath}
\usepackage{amscd}
\usepackage{amssymb}
\usepackage{amsfonts}
\newtheorem{theorem}{Theorem}[section]
\newtheorem{lemma}[theorem]{Lemma}

\newtheorem{corollary}[theorem]{Corollary}

\theoremstyle{definition}
\newtheorem{definition}[theorem]{Definition}

\theoremstyle{remark}
\newtheorem{remark}[theorem]{Remark}
\numberwithin{equation}{section}
\begin{document}
\title[On a class of complete non-compact gradient Yamabe solitons]
{On a class of complete non-compact gradient Yamabe solitons}
\author{Jia-Yong Wu}
\address{Department of Mathematics, Shanghai Maritime University,
Haigang Avenue 1550, Shanghai 201306, P. R. China}

\email{jywu81@yahoo.com}


\date{April 6, 2017}

\dedicatory{}
\begin{abstract}
We derive lower bounds of the scalar curvature on complete non-compact gradient
Yamabe solitons under some integral curvature conditions. Based on this, we prove
that potential functions of Yamabe solitons have at most quadratic growth for
distance function. We also obtain a finite topological type property on complete
shrinking gradient Yamabe solitons under suitable scalar curvature assumptions.
\end{abstract}
\maketitle
\section{Introduction and results}\label{sec1}
In this note we will study some scalar curvature and geometrical topology properties
on a large class of complete non-compact gradient Yamabe solitons.

\begin{definition}
Let $(M^n,g)$ be an $n$-dimensional complete Riemannian manifold and let $R$ denotes
the scalar curvature of $(M^n,g)$. It is said to be a \emph{complete non-compact
(resp. closed) Yamabe soliton} if manifold $(M^n,g)$ is complete non-compact (resp.
closed) and if it admits a smooth vector field $X$ satisfying
\begin{equation}\label{def1}
R g+\mathcal {L}_X g=\lambda g
\end{equation}
for some real constant $\lambda$, where $\mathcal {L}_X$ denotes the
Lie derivative with respect to the smooth vector $X$. A Yamabe
soliton is said to be \emph{shrinking}, \emph{steady} or
\emph{expanding} if $\lambda>0$, $\lambda=0$ or $\lambda<0$,
respectively.
\end{definition}
The vector field $X$ appearing in \eqref{def1} has necessarily to be
a conformal Killing vector field since the Yamabe flow
\begin{equation}\label{Yama}
\frac{\partial}{\partial t}g=-Rg,
\end{equation}
preserves the conformal structure. If $2X=\nabla f$ for some
smooth function $f$ on $M^n$, then we have a \emph{complete gradient
Yamabe soliton}
\begin{equation}\label{def2}
R g_{ij}+\nabla_i \nabla_j f=\lambda g_{ij},
\end{equation}
where function $f$ is often called the \emph{potential function}. When
$f$ is constant, the scalar curvature of Yamabe solitons becomes constant.

The Yamabe soliton is an important object in understanding the Yamabe flow, since
it can be regarded as the special solution of the Yamabe flow and naturally
arises as the limit of dilations of singularities in the Yamabe flow. Meanwhile,
the Yamabe soliton has inspired differential Harnack inequalities of the
Yamabe flow (see \cite{[Chow]}). The knowing the geometry of gradient Yamabe
solitons helps us to understand the geometry of singularities in the Yamabe flow.
It is known that \emph{any closed Yamabe soliton has constant scalar curvature}
(see \cite{[CD]} or \cite{[CLN]}, Appendix B). For the non-compact case, many
interesting results about non-compact Yamabe solitons have been studied in recent
papers, such as \cite{[CSZ]}, \cite{[CMM]}, \cite{[DaSe]}, \cite{[Ma-Ch]} and
\cite{[Ma-Vi]}.

On complete (closed or non-compact) Ricci solitons, B.-L. Chen \cite{[Chen]}
(or see \cite{[Zhang]}) obtained a uniform lower bound of scalar curvature
by using the cut-off function techinque in a local neighborhood of manifold.
Motivated by this result, it is natural to ask whether the scalar curvature
of complete non-compact Yamabe solitons shares a similar property? In general,
it seems to be difficult to answer this question. However, if some assumptions
are given, we can give some partial answer.
\begin{theorem}\label{bound}
Let $(M^n,g, f)$ be a complete non-compact gradient
Yamabe soliton of the form \eqref{def2}, whose Ricci curvature $Rc$ satisfies
\begin{equation}\label{qulvcon}
\lim_{r(x)\to\infty}\frac {1}{r(x)}\cdot\int^{r(x)}_1Rc(\gamma'(s),\gamma'(s))ds\geq0,
\end{equation}
where $r(x)$ is the distance function from
a fixed point $p\in M^n$ and $\gamma:[0,r(x)]\to
M^n$ is a unit speed minimal geodesic joining $p$ to $x$.
\begin{enumerate}
\item[(1)] If the gradient Yamabe soliton is  steady or shrinking,
then $R\geq0$.

\item[(2)] If the gradient Yamabe soliton is expanding,
then $R\geq\lambda$.
\end{enumerate}
\end{theorem}

\begin{remark}
From Section \ref{example} below, we see that if $(M,g,f)$ are
Einstein solitons with non-negative constant curavture,
or Gaussian Yamabe solitons, Theorem \ref{bound} obviously holds.
If $(M,g,f)$ is 2-dimensional steady cigar Yamabe soliton,
we can compute that $Rc=\frac{2}{1+r^2}g$ and hence
\eqref{qulvcon} holds and $R=\frac{4}{1+r^2}>0$.
\end{remark}

\begin{remark}
Since the curvature assumption in Theorem \ref{bound} obviously exclude
Einstein solitons with negative constant curvature. Can one remove or
weaken the assumption of the theorem?
\end{remark}

\begin{remark}
Recently L. Ma and V. Miquel have also proved similar scalar curvature
estimates under different curvature assumptions
(Theorems 1 and 2 in \cite{[Ma-Vi]}).
\end{remark}

In addition, as we all known, potential functions of complete gradient Ricci solitons
have at most quadratic growth for the distance function (see, e.g., \cite{[CaoZhou]}
or \cite{[CHD3]}). Inspired by the Ricci solitons case, based on Theorem
\ref{bound}, we can also derive similar results for potential functions
of complete gradient Yamabe solitons.
\begin{theorem}\label{bound2}
Under the same hypotheses of Theorem \ref{bound}, we have
\begin{enumerate}
  \item $f(x)\leq
\frac{\lambda}{2}\cdot r(x)^2+C_1\cdot r(x)+C_2$
  \quad\,$\mathrm{if}$ $\lambda\geq0$;

  \item $f(x)<
C_1\cdot r(x)+C_2$ \quad\quad\quad\quad\quad\quad$\mathrm{if}$ $\lambda<0$
\end{enumerate}
for some fixed constants $C_1$ and $C_2$. Here
$r(x):=d(p,x)$ is the distance function from some fixed point $p\in
M^n$.
\end{theorem}

In the following we obtain a topological result on complete shrinking
gradient Yamabe solitons with a proper upper bound on scalar curvature.
\begin{theorem}\label{topo}
Let $(M^n,g)$ be a complete Riemannian manifold satisfying
\[
Rg_{ij}+\nabla_i \nabla_j f\geq \lambda g_{ij},
\]
where $f\in C^\infty(M^n)$ and $\lambda>0$. Assume that the scalar curvature
$R\leq \mu$ for some constant $\mu<\lambda$. Then $M^n$ has finite topological type.
\end{theorem}

In particular, Theorem \ref{topo} implies
\begin{corollary}\label{cortop}
Any complete shrinking gradient Yamabe soliton of the form \eqref{def2} with
$R\leq \mu$ for some constant $\mu<\lambda$ has finite topological type.
\end{corollary}
\begin{remark}
Fang, Man and Zhang \cite{[FMZ]} proved that any
complete gradient shrinking Ricci soliton with bounded scalar
curvature has finite topological type.
\end{remark}

The structure of this paper is organized as follows. In Section
\ref{example}, we will give some simple examples of complete non-compact
Yamabe solitons. In Section \ref{bdsc}, we will employ the classical cut-off
technique to prove Theorem \ref{bound}. In Section \ref{potenfunc},
we apply Theorem \ref{bound} to prove Theorem \ref{bound2}. Finally, in
Section \ref{bdsc2}, we will prove Theorem \ref{topo} and Corollary \ref{cortop}.

\vspace{.1in}

\textbf{Acknowledgement}.
The preliminary version of the work was finished in 2011 (see arXiv:1109.0861).
For some reason, the author didn't submit it at that time. This work is
supported by Natural Science Foundation of Shanghai (17ZR1412800)
and National Natural Science Foundation of China (11671141).

\section{Examples}\label{example}

In this section, we give some simple examples of complete non-compact
Yamabe solitons. It is useful to consider these examples as models to
guide one's intuition about the structure of Yamabe solitons.

\vspace{0.3em}

\textbf{$\bullet$ Einstein soliton}

\vspace{0.3em}

Let $(M^n,g)$ be a complete non-compact Riemannian manifold with the
Ricci curvature $R_{ij}=\frac{\lambda}{n} g_{ij}$ and let the
potential function $f$ be a constant. Then we have
\[
R g_{ij}=\lambda g_{ij}.
\]
The above Yamabe soliton equation is said to be shrinking, steady or
expanding if $\lambda>0$, $\lambda=0$ or $\lambda<0$, respectively.

\vspace{0.3em}

\textbf{$\bullet$ Gaussian Yamabe soliton}

\vspace{0.3em}

The Gaussian Yamabe soliton on $\mathbb{R}^n$ is given by
\[
g_{ij}=\delta_{ij}\quad \mathrm{and}\quad f(x)=\frac
\lambda2|x|^2.
\]
Then we have $R=0$ and $\nabla_i\nabla_j f=\lambda\delta_{ij}$.
Therefore
\[
0=-Rg_{ij}=\nabla_i \nabla_j f-\lambda \delta_{ij}.
\]
Note that if $\lambda>0$ (or $\lambda<0$), then $(\mathbb{R}^n,
\delta_{ij}, \frac \lambda2|x|^2)$ is called the Gaussian shrinking
(or expanding) Yamabe soliton.

\vspace{0.3em}

\textbf{$\bullet$ Cigar Yamabe soliton}

\vspace{0.3em}

Hamilton's cigar soliton \cite{[Hamilton]} (see also \cite{[CHD],[CHD2]} or
Chapter 4 in \cite{[CLN]}) is the complete Riemannian surface, where
\[
g_{\small\Sigma}:=\frac{dx^2+dy^2}{1+x^2+y^2}.
\]
In polar coordinates: $x=\rho\cos\theta$ and $y=\rho\sin\theta$,  we may
rewrite
\[
g_{\Sigma}=\frac{d\rho^2+\rho^2d\theta^2}{1+\rho^2}.
\]
By a direct computation, the scalar curvature of $g_{\Sigma}$ is
$R=\frac{4}{1+\rho^2}$.
Let
\[
f:=-2\log(1+\rho^2).
\]
Then we have that
\[
Rg_{ij}+\nabla_i \nabla_j f=0.
\]
This is the equation of steady gradient Yamabe solitons in dimension two. But for high dimension case, it is not clear whether there are examples of non-compact steady gradient Yamabe solitons with rotationally symmetric metrics.


\section{Bounds of the scalar curvature}\label{bdsc}
Before we prove Theorem \ref{bound}, we first give a
useful lemma on the complete gradient Yamabe solitons.
\begin{lemma}\label{identity}
Let $(M^n, g)$ be a complete (closed or non-compact) Riemannian
manifold satisfying \eqref{def2}. Then we have the following identities
\begin{equation}\label{ident1}
\nabla_i R=\frac {1}{n-1}R_{ij}\nabla_j f
\end{equation}
and
\begin{equation}\label{ident2}
\Delta R-\frac{\langle\nabla R,\nabla f\rangle}{2(n-1)}+\frac
{R^2}{n-1}-\frac{\lambda R}{n-1}=0.
\end{equation}
\end{lemma}

\begin{proof}[proof of Lemma \ref{identity}]
Taking covariant derivatives of $Rg_{jk}+\nabla_j\nabla_k
f=\lambda g_{jk}$ yields
\[
\nabla_iRg_{jk}+\nabla_i\nabla_j\nabla_kf=0
\quad \mathrm{and}\quad
\nabla_jRg_{ik}+\nabla_j\nabla_i\nabla_kf=0.
\]
Taking the difference for above two equalities and applying the
commutating formula for covariant derivative, we get
\[
\nabla_iRg_{jk}-\nabla_jRg_{ik}+R_{ijkl}\nabla_l f=0.
\]
Then taking the trace on the indexes $j$ and $k$ for the above equality
gives \eqref{ident1}.

To prove \eqref{ident2}, taking the covariant derivatives to
\eqref{ident1} once again, and using the contracted second Bianchi,
we get
\begin{equation*}
\begin{aligned}
\Delta R&=\frac {1}{n-1}\nabla_i\left(R_{ij}\nabla_j f\right)\\
&=\frac {1}{n-1}\left(\frac 12\nabla_jR\nabla_j f+R_{ij}
\nabla_i\nabla_j f\right)\\
&=\frac {1}{n-1}\left[\frac 12\nabla_jR\nabla_j f+R_{ij}(\lambda
g_{ij}-Rg_{ij})\right].
\end{aligned}
\end{equation*}
This completes the proof of the lemma.
\end{proof}

Next using Lemma \ref{identity}, we can finish the proof of
Theorem \ref{bound}.

\begin{proof}[Proof of Theorem \ref{bound}]
To prove the result, we will apply a local technique of Li-Yau
\cite{[Li-Yau]} to obtain the scalar curvature estimate on
complete non-compact gradient Yamabe solitons.

Take a $C^2$-smooth cut-off function $\tilde{\varphi}$ on $[0,\infty)$
such that $\tilde{\varphi}(s)=1$ for $s\in[0,1]$, $\tilde{\varphi}(s)=0$
for $s\in[2,\infty)$, and $0\leq\tilde{\varphi}(s)\leq1$ for $s\in(1,2)$.
Furthermore, we take the derivatives of $\tilde{\varphi}$, satisfying
\begin{equation}\label{cutoff1}
-C\leq\frac{\tilde{\varphi}'(s)}{\tilde{\varphi}^{1/2}(s)}
\leq0\quad\mathrm{and}\quad\tilde{\varphi}''(s)\geq-C
\end{equation}
for some absolute constant $0<C<\infty$. Fix a point $p\in M$ and let
$r(x):=d(x,p)$ denotes the distance between $x$ and $p$ in $M^n$. Set
\[
\varphi(x)=\tilde{\varphi}\left(\frac{r(x)}{c}\right)
\]
for $c\in(0,\infty)$. Now consider the function $\Phi(x):=\varphi(x){\cdot}R(x)$
supported in $B_p(2c)$. By the argument of Calabi \cite{[Calabi]}
(see also Cheng and Yau \cite{[Cheng-Yau]}), without loss of generality
we assume $\Phi(x) \in C^2(M)$. Direct calculation yields
\[
\Delta \Phi=\varphi\cdot\Delta R+\frac
2c\tilde{\varphi}'\langle\nabla r,\nabla
R\rangle+\left(\frac{\tilde{\varphi}'}{c}\Delta
r+\frac{\tilde{\varphi}''}{c^2}\right)\cdot R.
\]
in $B_p(2c)$. Substituting  \eqref{ident2} into the right hand side
of the above equality,
\begin{equation*}
\begin{aligned}
\Delta \Phi&=\varphi\cdot\left[\frac{\langle\nabla R,\nabla
f\rangle}{2(n-1)}-\frac {R^2}{n-1}+\frac{\lambda R}{n-1}\right]
+\frac 2c\tilde{\varphi}'\langle\nabla r,\nabla R\rangle
+\left(\frac{\tilde{\varphi}'}{c}\Delta r
+\frac{\tilde{\varphi}''}{c^2}\right)\cdot R\\
&=\frac{\langle\nabla\Phi,\nabla f\rangle}{2(n-1)}
+\frac2c\cdot\frac{\tilde{\varphi}'}{\tilde{\varphi}}
\langle\nabla\Phi,\nabla r\rangle
+\frac{\tilde{\varphi}'}{c}\left[\Delta r-\frac{\langle\nabla
r,\nabla f\rangle}{2(n-1)}\right]\cdot R\\
&\,\,\,\,\,\,+\varphi\cdot\left(\frac{\lambda R}{n-1}-\frac
{R^2}{n-1}\right)+\frac{1}{c^2}\left[\tilde{\varphi}''-2\cdot
\frac{(\tilde{\varphi}')^2}{\tilde{\varphi}}\right]\cdot R
\end{aligned}
\end{equation*}
in $B_p(2c)$.

\textbf{Step 1:} If $R\geq 0$, then theorem follows.

\textbf{Step 2:} In the rest of proof, we only consider
the case: $R\leq 0$. Now we may assume that $\min_{x\in M}\Phi(x)<0$.
Then there exists a point
$x_0\in B_p(2c)$, such that
\[
\Phi(x_0)=\min_M\Phi(x)<0.
\]
This implies $R(x_0)<0$ and hence
\[
\nabla\Phi(x_0)=0\quad\mathrm{and}\quad\Delta\Phi(x_0)\geq0.
\]
Therefore at the point $x_0$, we have
\begin{equation}\label{keyineq}
\frac{\tilde{\varphi}'}{c}\cdot\left[\Delta r-\frac{\langle\nabla
r,\nabla f\rangle}{2(n-1)}\right]
+\varphi\cdot\left(\frac{\lambda}{n-1}-\frac
{R}{n-1}\right)+\frac{1}{c^2}\left[\tilde{\varphi}''-2\cdot
\frac{(\tilde{\varphi}')^2}{\tilde{\varphi}}\right]\leq0.
\end{equation}

\vspace{.1in}

Next we shall consider the nonexpanding and expanding cases
separately to study the key inequality \eqref{keyineq}.

\vspace{0.5em}

\noindent\textbf{(1) Nonexpanding gradient Yamabe soliton}

\vspace{0.5em}

For shrinking or steady gradient Yamabe solitons, i.e.,
$\lambda\geq0$, we consider two cases, depending on the location of
$x_0$.

\vspace{0.25em}

\textbf{Case (i)} Suppose $d (x_0,p)<c$, so that $\varphi=1$ in a
neighborhood of $x_0$. At this time $\tilde{\varphi}'=0$. Then
\eqref{keyineq} implies
\begin{equation*}
\begin{aligned}
0&\geq\lambda-R(x_0)
=-\Phi(x_0)+\lambda
\geq-\tilde{\varphi}\left(\frac{r(x)}{c}\right)\cdot R(x)+\lambda
\end{aligned}
\end{equation*}
for all $x\in M$. From above, we have
\[
R(x)\geq\lambda
\]
for all $x\in B_p(c)$, since
$\tilde{\varphi}\left(\frac{r(x)}{c}\right)=1$.

\vspace{0.5em}

\textbf{Case (ii)} Suppose $d (x_0,p)\geq c$. We still consider
\eqref{keyineq}. Below we want to get an
upper bound of $\Delta r-\frac{\langle\nabla r,\nabla
f\rangle}{2(n-1)}$. On one hand, by the second variation of distance,
\begin{equation}\label{vari}
\Delta r(x_0)\leq\int^{r(x_0)}_0\left[(n-1)(\eta'(s))^2
-\eta^2Rc(\gamma'(s),\gamma'(s))\right]ds
\end{equation}
for any unit minimal geodesic $\gamma:[0,r(x_0)]\to
M^n$ joining $p$ to $x_0$ and any piecewise $C^\infty$
function $\eta:[0,r(x_0)]\to [0, 1]$ satisfying $\eta(0)=0$ and
$\eta(r(x_0))=1$. On the other hand,
\begin{equation*}
\begin{aligned}
-\frac{\langle\nabla r,\nabla
f\rangle(x_0)}{2(n-1)}&=-\frac{\langle\nabla r,\nabla
f\rangle(p)}{2(n-1)}-\frac{1}{2(n-1)}\int^{r(x_0)}_0\kern-5pt
\gamma'(s)\langle\nabla f, \nabla r\rangle ds\\
&=-\frac{\langle\nabla r,\nabla
f\rangle(p)}{2(n-1)}-\frac{1}{2(n-1)}\int^{r(x_0)}_0\kern-5pt\nabla\nabla
f(\gamma'(s),\gamma'(s))ds\\
&=-\frac{\langle\nabla r,\nabla f\rangle(p)}{2(n-1)}-\frac{\lambda
r(x_0)}{2(n-1)}+\frac{1}{2(n-1)}\int^{r(x_0)}_0Rg
(\gamma'(s),\gamma'(s))ds,
\end{aligned}
\end{equation*}
where we used $\nabla r=\gamma'(s)$, $\nabla_{\gamma'(s)}\nabla r=0$
and \eqref{def2}. Combining this with \eqref{vari} and using
our assumption $R\leq0$ yields
\begin{equation}\label{vari4}
\begin{aligned}
\Delta r(x_0)-\frac{\langle\nabla r,\nabla f\rangle(x_0)}{2(n-1)}
&\leq(n-1)\int^{r(x_0)}_0(\eta'(s))^2ds-\frac{\langle\nabla r,\nabla
f\rangle(p)}{2(n-1)}\\
&\quad-\frac{\lambda r(x_0)}{2(n-1)}
-\int^{r(x_0)}_0\eta^2\cdot Rc(\gamma'(s),\gamma'(s))ds.
\end{aligned}
\end{equation}
To simply this inequality, we define the function $\eta(s)$
\begin{equation*}
\eta(s)=\left\{ \begin{aligned}
s\quad &&\mathrm{on}\,\, &[0,1], \\
1\quad &&\mathrm{on}\,\, &(1,r(x_0)].
\end{aligned} \right.
\end{equation*}
Plugging this cut-off function into \eqref{vari4} gives
\begin{equation}\label{vari5}
\begin{aligned}
&\Delta r(x_0)-\frac{\langle\nabla r,\nabla f\rangle(x_0)}{2(n-1)}
\leq(n-1)-\frac{\langle\nabla r,\nabla
f\rangle(p)}{2(n-1)}-\frac{\lambda r(x_0)}{2(n-1)}\\
&-\int^1_0s^2{\cdot}Rc(\gamma'(s),\gamma'(s))ds
-\int^{r(x_0)}_1Rc(\gamma'(s),\gamma'(s))ds\\
&\leq(n-1)-\frac{\langle\nabla r,\nabla
f\rangle(p)}{2(n-1)}-\frac{\lambda r(x_0)}{2(n-1)}
+C_0-\int^{r(x_0)}_1Rc(\gamma'(s),\gamma'(s))ds,
\end{aligned}
\end{equation}
where $C_0$ is some positive absolute constant.
Note that \eqref{cutoff1} implies
\[
\tilde{\varphi}''-2\cdot
\frac{(\tilde{\varphi}')^2}{\tilde{\varphi}}\geq-C,
\]
where $C$ is some positive absolute  constant and independent of $c$.
Substituting this inequality and \eqref{vari5}
into \eqref{keyineq}, we obtain\kern-12pt
\begin{equation*}
\begin{aligned}
\frac{\Phi(x)}{n-1}&\geq\frac{\Phi(x_0)}{n-1}\\
&\geq\frac{\tilde{\varphi}'(\frac{r(x_0)}{c})}{c}\Bigg[
(n-1)-\frac{\langle\nabla r,\nabla f\rangle(p)}{2(n-1)}
-\frac{\lambda r(x_0)}{2(n-1)}+C_0\\
&\quad\quad\quad\quad\quad\quad-\int^{r(x_0)}_1Rc(\gamma'(s),\gamma'(s))ds\Bigg]
+\frac{\lambda\varphi(x_0)}{n-1}-\frac{C}{c^2}.
\end{aligned}
\end{equation*}
Hence for all $x\in B_p(c)$,
\begin{equation}
\begin{aligned}\label{shizhi2}
\frac{R(x)}{n-1}&\geq\frac{\tilde{\varphi}'(\frac{r(x_0)}{c})}{c}
\Bigg[
(n-1)-\frac{\langle\nabla r,\nabla f\rangle(p)}{2(n-1)}
-\frac{\lambda r(x_0)}{2(n-1)}+C_0\\
&\quad\quad\quad\quad\quad\quad-\int^{r(x_0)}_1Rc(\gamma'(s),\gamma'(s))ds\Bigg]
+\frac{\lambda\varphi(x_0)}{n-1}-\frac{C}{c^2}.
\end{aligned}
\end{equation}
Note that from $d(x_0,p)\geq c$ and $\tilde{\varphi}(\frac{r(x_0)}{c})>0$,
we can judge that
\[
x_0\in B(p,2c)-B(p,c) \quad \mathrm{and}  \quad  1\leq\frac{r(x_0)}{c}<2.
\]
Since $-1-\epsilon\leq\tilde{\varphi}'\leq 0$ for any $\epsilon>0$
and $\lambda\geq 0$, \eqref{shizhi2} reduces to
\begin{equation}
\begin{aligned}\label{shizhi3}
\frac{R(x)}{n-1}&\geq-\frac{1+\epsilon}{c}\left[(n-1)+\frac{|\nabla
f|(p)}{2(n-1)}+C_0\right]-\frac{C}{c^2}\\
&\quad-\frac{\tilde{\varphi}'(\frac{r(x_0)}{c})}{c}
\cdot\int^{r(x_0)}_1Rc(\gamma'(s),\gamma'(s))ds.
\end{aligned}
\end{equation}
Under our assumption \eqref{qulvcon},
then by taking $c\rightarrow\infty$, since
$-1-\epsilon\leq\tilde{\varphi}'\leq 0$ and $c\leq r(x_0)<2c$
and $\lambda\geq 0$, from \eqref{shizhi3} we conclude that
$R(x)\geq 0$
for all $x\in M$ in Case (ii).

On the other hand, recall that Case (i), and we have $R(x)\geq\lambda$
for all $x\in B_p(c)$. By taking $c\to\infty$, we have that
$R(x)\geq\lambda$
for all $x\in M$ in Case (i).

Therefore in any case, as long as $\lambda\geq 0$, we have $R(x)\geq 0$
for all $x\in M$. But at the beginning of Step 2, we assume $R(x)\leq 0$
for all $x\in M$. Hence $R(x)\equiv 0$ for all $x\in M$ in Step 2.

In summary, combining Step 1, we finish the proof of the nonexpanding case.

\vspace{0.5em}

\noindent\textbf{(2) Expanding gradient Yamabe soliton}

\vspace{0.5em}

The proof of this case ($\lambda<0$) is similar to the arguments of
the nonexpanding case. First, we consider \textbf{Case (i)} i.e., $d(x_0,p)<c$. In the
same way, we still have $R(x)\geq\lambda$ for all $x\in B_p(c)$.

Then we consider \textbf{Case (ii)}, i.e., $d(x_0,p)\geq c$. At this time, (\ref{shizhi2})
still holds. Since $x_0\in B(p,2c)-B(p,c)$, we have
\begin{equation}\label{xianzh1f}
1\leq\frac{r(x_0)}{c}<2.
\end{equation}
We also know that $-1-\epsilon\leq\tilde{\varphi}'\leq 0$ and
$\lambda<0$. Combining this with \eqref{xianzh1f} and using the
fact: $\varphi\leq1$, we have by \eqref{shizhi2}
\begin{equation*}
\begin{aligned}
\frac{R(x)}{n-1}&\geq-\frac{1+\epsilon}{c}\Bigg[(n-1)+\frac{|\nabla
f|(p)}{2(n-1)}-\frac{c\lambda}{(n-1)}+C_0\Bigg]+\frac{\lambda}{n-1}-\frac{C}{c^2}\\
&\quad-\frac{\tilde{\varphi}'(\frac{r(x_0)}{c})}{c}
\cdot\int^{r(x_0)}_1Rc(\gamma'(s),\gamma'(s))ds
\end{aligned}
\end{equation*}
for all $x\in B_p(c)$. Noticing the assumption of our theorem \eqref{qulvcon},
and taking first $c\rightarrow\infty\,(c\leq r(x_0)<2c)$ and then
$\epsilon\rightarrow0$, we obtain
\begin{equation}\label{shizhi5}
R(x)\geq 2\lambda
\end{equation}
for all $x\in M$. Note that $R(x)\leq 0$ for all $x\in M$. Hence we conclude that
\[
2\lambda\leq R(x)\leq0
\]
for all $x\in M$ in Step 2.

\vspace{.1in}

In fact the lower bound estimate \eqref{shizhi5} is not optimal. In the following, we want to sharp  the estimate \eqref{shizhi5} for the expanding gradient Yamabe solitons
($\lambda<0$). To achieve it, it needs constructing a suitable $C^2$-smooth cut-off function on $[0,\infty)$. Here we mainly follow the argument of Zhang \cite{[Zhangsj]}. We first construct an non-negative piecewise linear function $\varphi(s)$ such that
\begin{equation}\label{linercut}
\varphi(s)=\left\{ \begin{aligned}
&&1\,\,\,\,\,\,&\mathrm{on}\,\, [0,1), \\
&&\frac{1+b-s}{b}\,\,\,\,\,\,&\mathrm{on}\,\, [1,b+1),\\
&&0\,\,\,\,\,\,&\mathrm{on}\,\, [b+1,\infty),
\end{aligned} \right.
\end{equation}
where $b>2$. It is easy to check that
\[
\varphi''-2\cdot
\frac{(\varphi')^2}{\varphi}\geq -C_1
\]
on $[0,\infty)$ for some positive absolute constant $C_1$. Since $-C_1<\varphi'\leq 0$
for all $x\in B_p(c)$, the inequality \eqref{shizhi2} becomes
\begin{equation}\label{shizhi2k}
\begin{aligned}
R(x)&\geq\frac{-C_1}{c}
\left[(n-1)^2+\frac{|\nabla f|(p)}{2}+(n-1)C_0+\frac{(n-1)C}{cC_1}\right]\\
&\quad-(n-1)\frac{\tilde{\varphi}'(\frac{r(x_0)}{c})}{c}
\int^{r(x_0)}_1Rc(\gamma'(s),\gamma'(s))ds
-\lambda \left(\frac{\varphi'r(x_0)}{2c}-\varphi(x_0)\right),
\end{aligned}
\end{equation}
where $x_0\in B_p((b+1)c)-B_p(c)$. When we let $c\to\infty\,(c\leq r(x_0)<(1+b)c)$,
combing our assumption \eqref{qulvcon}, the first and second terms of the right
hand side of \eqref{shizhi2k} tend to zero. So we only need to estimate the term
\[
\varphi'\left(\frac{r(x_0)}{c}\right)\frac{r(x_0)}{2c}
-\varphi\left(\frac{r(x_0)}{c}\right).
\]
Since $x_0\in B_p((b+1)c)-B_p(c)$, we have
\[
1\leq s_0=\frac{r(x_0)}{c}\leq1+b
\]
on $[1,b+1)$. Hence in fact we only need to estimate the function
$h_{\varphi}(s)$ on $[1,b+1)$, where $h_{\varphi}(s)$ is defined by
\[
h_{\varphi}(s):=\frac 12\varphi'(s)s-\varphi(s).
\]
For the above linear cut-off function in \eqref{linercut}, we see that
\[
h_{\varphi}(s)=-\frac{2b+2-s}{2b}
\]
for $s\in[1,b+1)$. Obviously, $h_{\varphi}(s)\geq-1-1/b$ for $s\in[1,2)$
and $h_{\varphi}(s)\geq-1$ for $s\in[2,b+1)$. Furthermore, for any small
positive number $\epsilon$, we can also construct a $C^2$ cut-off function
$\phi$ by smoothing the linear function $\varphi$ such that
\begin{equation*}
\phi(s)=\left\{ \begin{aligned}
&&\varphi(s)\,\,\,\,\,\,&\mathrm{on}\,\, [0,1)\cup[2,b)\cup[b+1,\infty),\\
&&\phi(s)\,\,\mathrm{satisfies}\,\, -\frac{1+\epsilon}{b}\leq\phi'(s)<0\quad&\mathrm{on}\,\, [1,2)\cup[b,b+1).
\end{aligned} \right.
\end{equation*}
So when $b$ is large and small positive number $\epsilon$,
we have
\[
h_{\phi}(s)\geq-1-\frac{1+\epsilon}{b}
\]
for $s\in[1,b+1)$. Now we regard this function $\phi(s)$ as the desired cut-off function.
Taking $c\to \infty$, $\epsilon\to 0$ and then $b\to \infty$, we get that
\[
h_{\phi}(s)=\frac 12\phi'(s)s-\phi(s)\geq-1
\]
for all $x\in M$. So, by \eqref{shizhi2k} we get $R(x)\geq\lambda$ for
all $x\in M$. Hence $\lambda\leq R(x)\leq0$
for all $x\in M$ in Step 2. At last, Theorem \ref{bound} follows by
combining Step 1.
\end{proof}

\section{Bounds of potential functions}\label{potenfunc}
In this section, we will discuss the potential function $f$ on a class of complete
non-compact Yamabe solitons. We mainly complete the proof of Theorem
\ref{bound2}.

\begin{proof}[Proof of Theorem \ref{bound2}]
(1) Firstly we discuss the steady or shrinking Yamabe gradient solitons, namely,
$\lambda\geq0$. By Theorem \ref{bound},
we have $R\geq0$ and hence
\begin{equation}\label{huajian1}
\nabla_i \nabla_j f=(\lambda-R)g_{ij}\leq \lambda g_{ij}.
\end{equation}
Given any point $x\in M$, let $\gamma:[0,d(x,p)]\rightarrow M$ be a unit speed
shortest geodesic joining $p$ to $x$. Consider the function $f(\gamma(s))$
and then we have for all $\bar{s}\in[0,r(x)]$ that the radial derivative
of $f$ is given by
\begin{equation}\label{radderiva}
\frac{d}{ds}\Big|_{s=\bar{s}}f(\gamma(s))
=\langle\nabla f,\gamma'(\bar{s})\rangle
=\int^{\bar{s}}_0(\nabla\nabla f)(\gamma'(s),\gamma'(s))ds
+\langle\nabla f,\gamma'(0)\rangle.
\end{equation}
Using \eqref{huajian1}, we conclude that
\[
\frac{d}{ds}\Big|_{s=\bar{s}}f(\gamma(s))
\leq\lambda \bar{s}+\langle\nabla f,\gamma'(0)\rangle.
\]
Integrating this along the geodesic $\gamma$, we obtain by taking $\bar{s}=r(x)$
\[
f(x)\leq
\frac{\lambda}{2}\cdot r(x)^2+C_1\cdot r(x)+C_2
\]
for some fixed constants $C_1$ and $C_2$.

(2) Secondly, we discuss the expanding case, i.e., $\lambda<0$. By
Theorem \ref{bound}, we know $R\geq\lambda$.
Hence we get
\[
\nabla_i \nabla_j f=cg_{ij}
\]
for some smooth function $c:=\lambda-R$, where $c\leq0$. Similar
to the computation of \eqref{radderiva}, we prove that
\[
\frac{d}{ds}\Big|_{s=\bar{s}}f(\gamma(s))
\leq0+\langle\nabla
f,\gamma'(0)\rangle.
\]
Integrating this inequality along the geodesic $\gamma$ implies that
\[
f(x)\leq C_1\cdot r(x)+C_2
\]
for some fixed constants $C_1$ and $C_2$. Hence we finish the proof
of Theorem \ref{bound2}.
\end{proof}

\section{The finite topological type}\label{bdsc2}
In this section, we will give a complete proof of the topology result
(Theorem \ref{topo}) on complete gradient Yamabe solitons.

\begin {proof}[Proof of Theorem \ref{topo}]
To prove the desired result, we only need to show that the potential
function $f$ is proper and has no critical points outside of a
compact set.

Now, we let $p\in M$ be a fix point and let $\gamma:[0,s]\rightarrow M$
be a shortest geodesic by arc length, with respect to some fixed
$g(t_0)$ (depend on the constant $\lambda$), jointing $p$ to any
point $x\in M$ (we only care about the point $x$ which are far away
from $p$). Note that $s:=d_{g(t_0)}(p,x)$. Then we have
\begin{equation}
\begin{aligned}\label{proeq1}
\langle\nabla f, \dot{\gamma}\rangle(x)&=\langle\nabla f,
\dot{\gamma}\rangle(p)+\int^s_0\frac{d^2}{ds^2}f(\gamma(s))ds\\
&\geq\langle\nabla f,\dot{\gamma}\rangle(p)+\int^s_0(\lambda-R)
g(\dot{\gamma},\dot{\gamma})ds\\
&\geq\lambda s-|\nabla f|(p)-\int^s_0 R
g(\dot{\gamma},\dot{\gamma})ds.
\end{aligned}
\end{equation}
Since $R\leq \mu$ for some constant $\mu<\lambda$, we have
\[
\int^s_0 R g(\dot{\gamma},\dot{\gamma})ds\leq \mu s.
\]
Plugging this into \eqref{proeq1}, we obtain
\begin{equation}\label{profequ3}
\langle\nabla f, \dot{\gamma}\rangle(x)\geq (\lambda-\mu) s-|\nabla f|(p).
\end{equation}
Noticing that $\lambda-\mu>0$, by \eqref{profequ3}, we get
\[
\frac{d}{ds}f(\gamma(s))=\langle\nabla f, \dot{\gamma}\rangle(x)
\geq(\lambda-\mu) s-|\nabla f|(p)
\]
i.e.,
\[
f(\gamma(s))\geq f(\gamma(0))+\frac {\lambda-\mu}{2}s^2-|\nabla
f|(p)\cdot s.
\]
Let $r(y):=d_{g(t_0)}(p,y)$, so that $s=r(x)$. We have that
\[
f(x)\geq f(p)+\frac{\lambda-\mu}{2}r^2-C_3 r
\]
for any $x\in M$. Obviously, $f^{-1}((-\infty,\tilde{c}])$ is compact for any
$\tilde{c}<\infty$. So $f$ is proper. We also observe that $f$ has no critical
points outside of a large compact domain (for example, we can choose a compact
domain $\bar{B}\left(p, \frac{2|\nabla f|(p)}{\lambda-\mu}\right)$).
\end{proof}

\bibliographystyle{amsplain}

\end{document}